%% file: CLT_circular28January,2018.tex
\documentclass{amsart}

\usepackage{amssymb}
\usepackage{amsmath}
\usepackage{amsthm}
\usepackage{amsbsy}
\usepackage{bm}
\usepackage{hyperref}
\date{\today}
\usepackage{cite}
\usepackage{array}

\input{mypreamble}

\title{Universality in the fluctuation of eigenvalues of random circulant matrices}

\author{Kartick Adhikari}
\address{Theoretical Statistics and  Mathematics Unit\\
        Indian Statistical Institute\\
        Kolkata 700108, India}

\email{kartickmath [at] gmail.com}

\author{Koushik Saha}
\address{Department of Mathematics\\
        Indian Institute of Technology Bombay\\
         Powai, Mumbai, Maharashtra 400076, India}

\email{koushik.saha [at] iitb.ac.in}

\date{\today}
\thanks{This work is partially supported by National Post-Doctoral Fellowship, India of Kartick Adhikari (reference no. PDF/2016/001601).}
\begin{document}

\begin{abstract}
We  show that the linear statistics of eigenvalues of random circulant matrices obey the Gaussian central limit theorem for a large class of input sequences.
\end{abstract}

\maketitle

\section{ introduction and main results}
Let $A_n$ be an $n\times n$ matrix with real or complex entries. The linear statistics of
eigenvalues $\lambda_1,\lambda_2,\ldots, \lambda_n$ of $A_n$ is a
function of the form
$$
\frac{1}{n}\sum_{k=1}^{n}f(\lambda_k)
$$
where $f$ is some  fixed function. The function $f$ is known as the test function. One of the interesting 
object  to study in random matrix theory  is the fluctuation of linear
statistics of eigenvalues of random matrices.  The study of fluctuation of linear statistics of eigenvalues was  initiated by Arharov \cite{arharov} in 1971 for sample covariance matrices. In 1975 Girko \cite{girko} studied the central limit theorem (CLT)  of the traces of the Wigner and sample covariance matrices using martingale techniques. In
1982, Jonsson \cite{jonsson} proved the 
CLT of linear eigenvalue statistics for Wishart matrices using method of moments.  After that the fluctuations of
eigenvalues for various  random matrices have been extensively studied by various people.  
%The CLT for the linear statistics of eigenvalues of various matrix models has been studied by various people 
For new results on fluctuations  of linear eigenvalue statistics of  Wigner and sample covariance matrices,  see  \cite{johansson1998}, \cite{soshnikov1998tracecentral}, \cite{bai2004clt}, \cite{lytova2009central}, \cite{shcherbina2011central}.  For   band and sparse  random matrices, see  \cite{anderson2006clt},  \cite{jana2014}, \cite{li2013central},  \cite{shcherbina2015} and for Toeplitz and band Toeplitz matrices, see  \cite{chatterjee} and \cite{liu2012}. 

%band Wishart matrices
%\cite{jana2014}, \cite{shcherbina2015}, Toeplitz matrices
%\cite{liu2012}, band Toeplitz matrices \cite{chatterjee} ).
%Recently, Shcherbina \cite{shcherbina2015} proved the CLT for
%linear eigenvalue statistics of the smooth test functions for some
%classical models of random matrix theory: deformed Wigner and
%sample covariance matrices, sparse matrices, diluted random
%matrices, matrices with heavy tales etc.
%\cite{meckes}

In a recent article \cite{adhikari_saha2017},  the CLT for linear eigenvalue statistics has been established in total variation norm for  the circulant matrices and of its variants  with Gaussian entries. Here we consider the fluctuation problem for the circulant matrices with general entries which are independent and  satisfy some moment condition.
A sequence is said to be an {\it input sequence} if the matrices are constructed from the given sequence. We consider the input sequence of the form $\{x_i: i\ge 0\}$
%The circulant is constructed from this given input sequence, 
and the circulant matrix is defined as 
$$
C_n=\left(\begin{array}{cccccc}
x_0 & x_1 & x_2 & \cdots & x_{n-2} & x_{n-1} \\
x_{n-1} & x_0 & x_1 & \cdots & x_{n-3} & x_{n-2}\\
\vdots & \vdots & {\vdots} & \ddots & {\vdots} & \vdots \\
x_1 & x_2 & x_3 & \cdots & x_{n-1} & x_0
\end{array}\right).
$$
For $j=1,2,\ldots, n-1$, its $(j+1)$-th row is obtained by giving its $j$-th row a right circular shift by one positions and the (i,\;j)-th element of the matrix  $x_{(j-i) \mbox{ \tiny{mod} } n}$.
% we write $x_n$ instead of $x_0$.

In our first result  we consider the fluctuation of linear  eigenvalue statistics of  circulant matrices with a  polynomial test function. Let $P_d(x)=\sum_{k=2}^da_kx^k$ be a real polynomial of degree $d$ where $d\geq 2$. 

\begin{theorem}\label{thm:cirpoly} 
Suppose  $C_n$ is the random circulant matrix with independent input sequence $\{\frac{X_i}{\sqrt n}\}_{i\geq 0}$ such that 
\begin{equation}\label{eqn:condition}
\E(X_i)=0, \E(X_i^2)=1 \ \mbox{and}\  \sup_{i\geq 0}\E(|X_i|^k)=\alpha_k<\infty \ \mbox{for}\ k\geq 3.
\end{equation}  
Then, as $n\to \infty$,
\begin{align*}
\frac{\Tr [P_d(C_n)]-\E\Tr [P_d(C_n)]}{\sqrt{n}}\stackrel{d}{\longrightarrow} N(0,\sigma_{p_d}^2), 
\end{align*}
where $\sigma_{p_d}^2= \sum_{\ell=2}^d a_{\ell}^2{\ell}!\sum_{s=0}^{{\ell}-1}f_{{\ell}}(s)$ and $f_{\ell}(s)=\sum_{k=0}^{s}(-1)^k\binom{\ell}{k}(s-k)^{\ell-1}$.
\end{theorem}
We  use the method of moments to prove Theorem \ref{thm:cirpoly}.  Note that, the constant term and the first degree term are not considered in the polynomial $P_d(x)$. The constant term of a matrix polynomial will  be a constant times the identity matrix and this term will not effect the fluctuation result of linear eigenvalue  statistics of a random matrix as we are centering  the linear eigenvalue statistics by its mean. 

If we consider a degree one monomial  of 
%a circulant matrix, that is, 
the random circulant matrix $C_n$ with independent input sequence $\{\frac{X_i}{\sqrt n}\}_{i\geq 0}$ where $\E(X_i)=0$ then 
$$\frac{\Tr(C_n)-\E\Tr(C_n)}{\sqrt n}=X_0.$$
Thus the limiting distribution depends on the distribution of $X_0$ and hence CLT type result does not hold for degree one monomial. Due to these reasons we have not considered the constant and the first degree terms in $P_d(x)$.

Next we consider the fluctuation problem for  circulant matrices  in total variation norm. It has been shown \cite{adhikari_saha2017} that 
$$
\frac{\Tr(A_n^{p_n})-\E(\Tr(A_n^{p_n}))}{\sqrt{\var(\Tr(A_n^{p_n}))}} \;\mbox{converges in total variation norm to } N(0,1),
$$
as $n\to \infty$, where $p_n=o(\log n/\log\log n)$ and $A_n$ is one of circulant, reverse circulant, symmetric circulant and Hankel matrices with Gaussian inputs. In this article, we show that the above results hold  when the matrices are constructed from the input sequence belongs to $\mathcal L(c_1,c_2)$, for some $c_1,c_2>0$ and subgaussian. 

\begin{definition}
For each $c_1,c_2>0$, let $\mathcal L(c_1,c_2)$ be the class of probability measures on $\mathbb R$ that arise as laws of random variables like $u(Z)$, $Z$ is a standard Gaussian random variable and $u$ is a twice continuously differentiable function such that for all $x\in \mathbb R$
$$
|u'(x)|\le c_1 \;\;\mbox{ and } |u''(x)|\le c_2.
$$ 
\end{definition}

For example, the standard Gaussian random variable is in $\mathcal L(1,0)$. The uniformly distributed random variable on $[0,1]$ is in $\mathcal L((2\pi)^{-1/2}, (2\pi e)^{-1/2})$. 

\begin{definition}
A random variable $X$ is said to be {\it $\sigma$-subgaussian} or subgaussian with parameter $\sigma$, $\sigma>0$, if $
\E[e^{tX}]\le e^{\sigma^2t^2/2}
$
 for every $t\in \R$. 
\end{definition}
 For example, the Bernoulli random variable with mass at $+1$ and $-1$ with equal probability is $1$-subgaussian. More generally, if $X$ is a random variable with $\E [X]=0$ and  $|X|\le \sigma$ for some $\sigma>0$, then $X$ is $\sigma$-subgaussian. The normal random variable with mean zero variance $\sigma^2$ is $\sigma$-subgaussian. 
 
Also note that if a random variable $X$ is $\sigma$-subgaussian, then its (absolute) moments
are bounded above by an expression involving $\sigma$ and
the gamma function (see e. g. \cite[p. 93]{stroock}). Therefore if a sequence of random variables $\{X_i\}_{i\geq 0}$ is $\sigma$-subgaussian then $\sup_{i\geq 0}E(|X_i|^k)<\infty$ for all $k\in \mathbb N$. We use this fact in the proof of  Theorem \ref{thm:main1}. Now we have the following central limit theorem result in total variation norm.

%It should be noted that for a $\sigma$-subgaussian random variable we have a bound on all moments that is in explicit function of $\sigma$. Because in the proof of Theorem 4 we use the fact that $\sup_{j\geq 0}E(|X_j|^k)<\infty$, which in Theorem 1 is given. 
% 

\begin{theorem}\label{thm:main1}
Suppose  $C_n$ is the random circulant matrix with input sequence $\{\frac{X_i}{\sqrt n}\}_{i\geq 0}$ such that $X_i$'s are independent symmetric $\sigma$-subgaussian random variables and  $X_i\in \mathcal L(c_1,c_2)$ for some finite $c_1$ and $c_2$.  Then, as $n\to \infty$,
\begin{equation}\label{total variation result}
\frac{\Tr(P_d(C_n))-\E(\Tr(P_d(C_n)))}{\sqrt{Var(\Tr(P_d(C_n)))}} \;\mbox{converges in total variation to } N(0,1),
\end{equation}
where $P_d(x)=\sum_{k=2}^da_kx^k$, a real polynomial of degree $d\ge 2$.
\end{theorem}
\begin{remark}
 As we are dealing with  circulant matrices in this article,   we have stated the total variation norm convergence result for  circulant matrices only. But the result \eqref{total variation result} holds  for other variants  of circulant matrices also, namely, for reverse circulant  and symmetric circulant matrices. For description of these matrices, see \cite{adhikari_saha2017}.
\end{remark}
Note that there is a large class of random variables which satisfy the assumptions on the input sequence in Theorem \ref{thm:main1}.  For example, standard Gaussian random variable,  symmetric uniform random variable and linear combination of these two belong to $\mathcal L(c_1,c_2)$ for some $c_1,c_2\geq 0$ and subgaussian. The proof  techniques of Theorem \ref{thm:main1} passively depend  on Stein's method and second order Poincar\'e inequality (see \cite{chatterjee}). In particular, we use Result \ref{re:sourav}, which relies on Stein's method and second order Poincar\'e inequality.
% for the proof of Theorem \ref{thm:main1}. 
The rest of the article is organized as follows. In Section \ref{sec:poly} we give a proof of Theorem \ref{thm:cirpoly} using moment method. In Section \ref{sec:totalvariation}, we prove  Theorem \ref{thm:main1}.

\section{Proof of Theorem \ref{thm:cirpoly}}\label{sec:poly}

%In  this section we prove Theorem \ref{thm:cirpoly} by the method of moments. We have the following corollary for monomials.
%
%\begin{corollary}\label{cor:monomial}
%Let $C_n$ be the circulant matrix with input sequence $\{X_n\}$ satisfying condition \eqref{eqn:condition}. Then for fixed positive integer $d\geq 2$, as $n\to \infty$,
%$$
%\frac{\Tr(C_n^{d})-\E\Tr(C_n^{d})}{\sqrt{n^{d+1}}} \longrightarrow N(0,\sigma_d^2),
%$$
% where $\sigma_d^2=d!\sum_{s=1}^df_d(s)$.
%\end{corollary}

We first define some notation which will be used   in the proof of Theorem \ref{thm:cirpoly}.
\begin{align}
A_p&=\{(i_1,\ldots,i_p)\in \Z^p\suchthat i_1+\cdots+i_p=0\;{(\mbox{mod $n$})},\; 0\le i_1,\ldots, i_p\le n-1\},\label{def:A_p}\\
A_{p}'&=\{(i_1,\ldots,i_p)\in \Z^p\suchthat i_1+\cdots+i_p=0\;{(\mbox{mod $n$})},\; 0\le i_1\neq i_2\neq \cdots\neq i_p\le n-1\},\nonumber\\
A_{p,s}&=\{(i_1,\ldots,i_p)\in \Z^p\suchthat i_1+\cdots+i_p=sn,\; 0\le i_1,\ldots, i_p\le n-1\},\nonumber \\
A_{p,s}'&=\{(i_1,\ldots,i_p)\in \Z^p\suchthat i_1+\cdots+i_p=sn,\; 0\le i_1\neq i_2\neq \cdots\neq i_p\le n-1\}.\nonumber
\end{align}
We prove Theorem \ref{thm:cirpoly} by the method of moments. To apply this method  we need to calculate the higher order moments of linear eigenvalue statistics of the circulant matrices and  that involve the trace of higher power of the circulant matrices.  So we calculate the trace of $(C_n)^p$ for some positive integer $p$.

Let $e_1,\ldots,e_n$ be the standard unit vectors in $\mathbb R^n$, i.e., $e_i=(0,\ldots,1,\ldots, 0)^t$ ($1$ in $i$-th place).  Therefore we have 
\begin{align*}
(C_n)e_i=\mbox{$i$-th column}=\sum_{i_1=0}^{n-1}x_{i_1}e_{i-i_1 \mbox{ mod $n$}},
\end{align*}
for $i=1,\ldots, n$. In the last equation $e_0$ stands for $e_n$. Repeating the procedure we get 
\begin{align*}
(C_n)^2e_i=\sum_{i_1,i_2=0}^{n-1}x_{i_1}x_{i_2}e_{i-i_1-i_2 \mbox{ mod $n$}},
\end{align*}
for $i=1,\ldots, n$. Therefore in general we get
\begin{align*}
(C_n)^{p}e_i&=\sum_{i_1,\ldots,i_{p}=0}^{n-1}x_{i_1}\ldots x_{i_{p}}e_{i-i_1-i_2-i_3\cdots -i_p \mbox{ mod $n$}},
\end{align*}
for $i=1,\ldots, n$.  Therefore the trace of $C_n^{p}$ can be written as 
\begin{align}\label{trace formula C_n}
\Tr(C_n^{p})=\sum_{i=1}^{n}e_i^t(C_n)e_i=n\sum_{A_{p}}x_{i_1}\cdots x_{i_{p}},
\end{align}
where $A_p$ is as defined in \eqref{def:A_p}. For a similar result on the trace of band Toeplitz matrix see \cite{liu_wang2011}.
The following result will be used in the proof of  Theorem \ref{thm:cirpoly}.
\begin{result}\label{ft:variance}
Consider  $A_{p}$ as defined above. Then 
\begin{align*}
\lim_{n\to \infty}\frac{|A_p|}{n^{p-1}}=\sum_{s=0}^{p-1}\lim_{n\to \infty}\frac{|A_{p,s}|}{n^{p-1}}=\sum_{s=0}^{p-1}f_p(s),
\end{align*}
where $$f_p(s)=\frac{1}{(p-1)!}\sum_{k=0}^{s}(-1)^k\binom{p}{k}(s-k)^{p-1}.$$
\end{result}
For the proof of Result \ref{ft:variance}, we refer to \cite[Lemma 13]{adhikari_saha2017}.
Assuming the lemma we proceed to prove Theorem \ref{thm:cirpoly}.

\begin{proof}[Proof of Theorem \ref{thm:cirpoly}]
We first  calculate expected value of $\Tr[P_d(C_n)]$. 
 Using the trace formula \eqref{trace formula C_n},   we get 
\begin{align*}
\E(\Tr[P_d(C_n)])&=\sum_{k=2}^da_k\E\Tr[C_n^k]=\sum_{k=2}^d \frac{a_k}{n^{\frac{k}{2}-1}} \sum_{A_k}\E[X_{i_1}\cdots X_{i_k}].
\end{align*}
Note that, for $\E[X_{i_1}\cdots X_{i_k}]$ to be non-zero, each random variable has to appear at least twice as the random variables have mean zero. Again the index variables satisfy one constrain since $(i_1,i_2,\ldots,i_k)$ belongs to $A_k$. Thus we have at most  $(\frac{k}{2}-1)$  free choice in the index set. Due to this fact and \eqref{eqn:condition}, we have 
\begin{align}\label{eqn:mean}
\E(\Tr[P_d(C_n)])=O(1).
\end{align} 
Now  we calculate the limit of the variance of $\frac{\Tr[P_d(C_n)]-\E(\Tr[P_d(C_n)])}{\sqrt n}$. This variance calculation will help us to understand the behaviour of higher order central moments of  $\Tr[P_d(C_n)]$ as $n$ tends to infinity. 
%Later we calculate the limit of the higher moments it. 
By \eqref{eqn:mean} we have
\begin{align*}
\lim_{n\to \infty}\var\l(\frac{\Tr[P_d(C_n)]-\E(\Tr[P_d(C_n)])}{\sqrt n}\r)=\lim_{n\to \infty}\frac{1}{n}\E(\Tr[P_d(C_n)])^2.
\end{align*}
Expanding the polynomial $P_d$ and using the trace formula \eqref{trace formula C_n}, we have 
\begin{align}\label{eqn:var}
\frac{1}{n}\E(\Tr[P_d(C_n)])^2&=\sum_{i_1,i_2=2}^d a_{i_1}a_{i_2}\frac{1}{n^{\frac{i_1+i_2}{2}-1}}\sum_{A_{i_1}, A_{i_2}}\E[X_{j_1}\cdots X_{j_{i_1}}X_{k_1}\cdots X_{k_{i_2}}]\nonumber
\\&=\sum_{i_1,i_2=2}^da_{i_1}a_{i_2}\frac{1}{n^{\frac{i_1+i_2}{2}-1}}\sum_{s=0}^{i_1-1}\sum_{t=0}^{i_2-1}\sum_{A_{i_1,s}, A_{i_2,t}}\E[X_{j_1}\cdots X_{j_{i_1}}X_{k_1}\cdots X_{k_{i_2}}].
\end{align}
Note that, for the non-zero contribution, no random variable can appear only once, as the random variables are independent and have  zero mean. Therefore each indices in $\{j_1,\ldots,j_{i_1},k_1,\ldots, k_{i_2}\}$ has to  appear at least twice. Observe that, if there is a self-matching in $\{j_1,\ldots, j_{i_1}\}$ or in $\{k_1,\ldots, k_{i_2}\}$, then the indices satisfy at least two equations. Therefore in such cases we have  $|A_{i_1,s}|| A_{i_2,t}|=O(n^{\frac{i_1+i_2}{2}-2})$. As  all the  moments of the input random variables are finite by \eqref{eqn:condition},  we have
$$
\sum_{A_{i_1,s}, A_{i_2,t}}\E[X_{j_1}\cdots X_{j_{i_1}}X_{k_1}\cdots X_{k_{i_2}}]=O(n^{\frac{i_1+i_2}{2}-2}),
$$   
when $A_{i_1,s}, A_{i_2,t}$ satisfy the self matching condition. Therefore the maximum contribution comes when $\{j_1,\ldots, j_{i_1}\}$ matched with $\{k_1,\ldots, k_{i_2}\}$ completely. This is possible only when $i_1=i_2$ and $s=t$, otherwise there will be a self-matching either in $\{j_1,\ldots, j_{i_1}\}$ or in $\{k_1,\ldots, k_{i_2}\}$. Thus, from \eqref{eqn:var}, we get 
\begin{align*}
\lim_{n\to \infty}\frac{1}{n}\E(\Tr[P_d(C_n)])^2=\lim_{n\to \infty}\sum_{i=2}^d a_{i}^2\  i! \frac{1}{n^{i-1}}\sum_{s=0}^{i-1}\sum_{A_{i,s}}\E[X_{j_1}^2\cdots X_{j_{i}}^2].
\end{align*}
The factor $i!$ appeared because $\{k_1,\ldots, k_{i}\}$ can match with given vector $(j_1,j_2,\ldots, j_{i})$ in $i!$ ways. 
%Note that the number of cases  for which at least one equality holds in the indices of $(i_1,\ldots,i_p)\in A_{p,s}$  is $O(n^{p-2})$.
%The factor $i!$ appeared because $(j_1,\ldots, j_{i})$ can match with $\{k_1,\ldots, k_{i}\}$ in $i!$ many ways. 
The maximum contribution comes when $(j_1,\ldots, j_{i})$ consists of distinct elements and that contribution is $O(n^{i-1})$. Otherwise the contribution will be of the order of $O(n^{i-2})$. Therefore we have 
\begin{align*}
\lim_{n\to \infty}\frac{1}{n}\E(\Tr[P_d(C_n)])^2&=\lim_{n\to \infty}\sum_{i=2}^da_{i}^2\  i! \frac{1}{n^{i-1}}\sum_{s=0}^{i-1}\sum_{A_{i,s}'}\E[X_{j_1}^2\cdots X_{j_{i}}^2]
\\&=\sum_{i=2}^da_{i}^2\  i!  \sum_{s=0}^{i-1}\lim_{n\to \infty}\frac{|A_{i,s}'|}{n^{i-1}}
=\sum_{i=2}^da_{i}^2\  i!  \sum_{s=0}^{i-1}\lim_{n\to \infty}\frac{|A_{i,s}|}{n^{i-1}},
\end{align*}
where $A_{i,s}'$ and $A_{i,s}$ are as defined in \eqref{def:A_p}.  The last equality holds because if any two indices of $(j_1,\ldots, j_{i})$ are equal then $|A_{i,s}|=O(n^{i-2})$, which contribute zero in the limit. Therefore from Result \ref{ft:variance}, we get
\begin{align*}
\lim_{n\to \infty}\frac{1}{n}\E(\Tr[P_d(C_n)])^2=\sum_{i=2}^da_{i}^2\  i! \sum_{s=0}^{i-1}f_i(s)
\end{align*}
Thus the limiting variance $\sigma_{P_d}^2$ is given by
\begin{align}\label{liming variance}
\sigma_{p_d}^2=\lim_{n\to \infty}\var\l(\frac{\Tr[P_d(C_n)-\E(\Tr[P_d(C_n)])}{\sqrt n}\r)=\sum_{i=2}^da_{i}^2\  i! \sum_{s=0}^{i-1}f_i(s).
\end{align}
Next we calculate the higher order  moments of $\frac{\Tr[P_d(C_n)]-\E(\Tr[P_d(C_n)])}{\sqrt n}$. Using the binomial expansion we have 
\begin{align}\label{eqn:expansion}
\left(\frac{\Tr [P_d(C_n)]-\E\Tr [P_d(C_n)]}{\sqrt{n}}\right)^k&=\frac{1}{n^{\frac{k}{2}}}\sum_{j=0}^{k}(-1)^{k-j}\binom{k}{j}(\Tr [P_d(C_n))^j](\E\Tr [P_d(C_n)])^{k-j}.
\end{align} 
%Note that \eqref{eqn:mean} implies that $L$ does not depend on $n$. 
Since $\E\Tr [P_d(C_n)]=O(1)$ (see \eqref{eqn:mean}), we focus on $(\Tr [P_d(C_n)])^j$.  By expanding the polynomial  we get
\begin{align}\label{eqn:exp2}
(\Tr [P_d(C_n)])^j=\sum_{I_j}a_{i_1}a_{i_2}\ldots a_{i_j}[\Tr C_n^{i_1}\cdots \Tr C_n^{i_j}],
\end{align}
where $I_j=\{(i_1,\ldots , i_{j})\; :\; 2\le i_1,\ldots, i_j\le d\}$. From the trace formula \eqref{trace formula C_n} of the circulant matrix, we have 
\begin{equation}\label{eqn_referee}
\E[\Tr C_n^{i_1}\cdots \Tr C_n^{i_j}]=\frac{1}{n^{\frac{i_1+\cdots+i_j}{2}-j}}\sum_{A^{(i_1,\ldots, i_j)}}\E\l(\prod_{\ell=1}^{j}[X_{k_{\ell,1}}\cdots X_{k_{\ell,i_{\ell}}}]\r),
\end{equation}
where $A^{(i_1,\ldots, i_j)}=\{(A_{i_1},\ldots, A_{i_j}):\; 2\le i_1,\ldots, i_j\le d\}$ and $A_{i_1},\ldots,A_{i_j}$ are as defined in \eqref{def:A_p}. Also note that in the sum in the right hand side of \eqref{eqn_referee} for each $\ell$, we have $(k_{\ell,1},k_{\ell,2},\ldots,k_{\ell,i_\ell})\in A_{i_\ell}$. For non zero contribution, the each random variables in $\{X_{k_{\ell,1}},\ldots ,X_{k_{\ell,i_{\ell}}}\; :\; \ell=1,\ldots, j\}$ must occur at least twice as the random variables have mean zero.  Observe that, following the  arguments given in variance calculation, we get the maximum contribution when for every $\ell$ there exists $\ell'$ such that $i_{\ell}=i_{\ell '}$ and the sets $\{{k_{\ell,1}},\ldots ,{k_{\ell,i_{\ell}}}\}$  and  $\{{k_{\ell',1}},\ldots ,{k_{\ell',i_{\ell'}}}\}$ are same with distinct elements. Therefore we need a pair matching in $\{i_1,\ldots,i_j\}$ to have maximum contribution. In other cases we have lower order contribution, as  all the moments of the random variables are finite. Thus we get 
\begin{align}\label{eqn:lesscontribution}
\sum_{A^{(i_1,\ldots, i_j)}}\E\l(\prod_{\ell=1}^{j}[X_{k_{\ell,1}}\cdots X_{k_{\ell,i_{\ell}}}]\r)=O(n^{\frac{i_1+\cdots+i_{j}}{2}-\lceil \frac{j}{2}\rceil}).
\end{align}
Therefore using \eqref{eqn:lesscontribution}, from \eqref{eqn:exp2} we get
\begin{align}\label{eqn:moment}
\E(\Tr [P_d(C_n)])^j=O(n^{j-\lceil \frac{j}{2}\rceil}).
\end{align}
Therefore using \eqref{eqn:mean} and \eqref{eqn:moment}, from \eqref{eqn:expansion} we get 
\begin{align*}
\lim_{n\to \infty}\E\left(\frac{\Tr [P_d(C_n)]-\E\tr [P_d(C_n)]}{\sqrt{n}}\right)^k=0, \;\;\mbox{when $k$ is odd}.
\end{align*}
Next we calculate the even moments. We use $2k$ instead of $k$. Again due to \eqref{eqn:mean} and \eqref{eqn:moment}, from \eqref{eqn:expansion} we get 
\begin{align*}
&\lim_{n\to \infty}\E\left(\frac{\Tr [P_d(C_n)]-\E\Tr [P_d(C_n)]}{\sqrt{n}}\right)^{2k}=\lim_{n\to \infty}\frac{1}{n^k}\E(\Tr [P_d(C_n)])^{2k}
\\=&\frac{(2k)!}{k!2^k}\sum_{I_k}a_{i_1}^2\cdots a_{i_k}^2 \lim_{n\to \infty}\frac{i_1!\cdots i_k!}{n^{i_1+\cdots+i_k-k}}\sum_{A^{(i_1,\ldots, i_k)}}\E\l[\prod_{\ell=1}^{k}[X_{k_{\ell,1}}^2\cdots X_{k_{\ell,i_{\ell}}}^2]\r].
\end{align*}
The factor $\frac{(2k)!}{k!2^k}$ appear because that many pair matched possible among $2k$ variables $\{i_1,\ldots, i_{2k}\}$. 
%We rename the indices as $\{i_1,\ldots, i_k\}$, due to the pair matching in $\{i_1,\ldots, i_{2k}\}$. 
After the pair matching in $\{i_1,\ldots, i_{2k}\}$, we rename the indices as $\{i_1,\ldots, i_k\}$. The factor $i_1!\cdots i_k!$ appear because,  for $\ell=1,\ldots, k$, each vector  $(k_{\ell,1},\ldots,k_{\ell,i_{\ell}})$  can be pair matched  with $\{k'_{\ell,1},\ldots,k'_{\ell,i_{\ell}}\}$ in  $i_{\ell}!$ many ways.  Now we have 
\begin{align*}
\lim_{n\to \infty}\frac{1}{n^{i_1+\cdots+i_k-k}}\sum_{A^{(i_1,\ldots, i_k)}}\E\l[\prod_{\ell=1}^{k}[X_{k_{\ell,1}}^2\cdots X_{k_{\ell,i_{\ell}}}^2]\r]=\lim_{n\to \infty}\frac{|A^{(i_1,\ldots, i_k)'}|}{n^{i_1+\cdots+i_k-k}},
\end{align*}
where $A^{(i_1,\ldots, i_k)'}=\{(A_{i_1}',\ldots, A_{i_k}')\;:\; \mbox{all coordinates are distinct throughout all } A_{i_l}'\}$ and $A_{i_l}', 1\leq l\leq k$ are as in \eqref{def:A_p}.  Again we have
$$
\lim_{n\to \infty}\frac{|A^{(i_1,\ldots, i_k)'}|}{n^{i_1+\cdots+i_k-k}}=\lim_{n\to \infty}\frac{|A^{(i_1,\ldots, i_k)}|}{n^{i_1+\cdots+i_k-k}}=\prod_{\ell=1}^k\lim_{n\to \infty}\frac{|A_{i_{\ell}}|}{n^{i_{\ell}-1}}.
$$
Therefore by Result \ref{ft:variance}, we get 
\begin{align}\label{2kth moment of normal}
\lim_{n\to \infty}\E\left(\frac{\Tr [P_d(C_n)]-\E\Tr [P_d(C_n)]}{\sqrt{n}}\right)^{2k}&=\frac{(2k)!}{k!2^k}\sum_{I_k}\prod_{\ell=1}^k a_{i_{\ell}}^2\  i_{\ell}! \sum_{s=0}^{i_{\ell}-1}f_{i_{\ell}}(s) \nonumber\\
&=\frac{(2k)!}{k!2^k} \l( \sum_{i=2}^d a_{i}^2\ i! \sum_{s=0}^{i-1}f_{i}(s)\r)^k,
\end{align}
where $I_k=\{(i_1,\ldots , i_{k})\; :\; 2\le i_1,\ldots, i_k\le d\}$. The  final expression in \eqref{2kth moment of normal}  is the $2k$-th moment of $N(0,\sigma_{p_d}^2)$ and this completes the proof.
\end{proof}

\section{proof of theorem \ref{thm:main1}}\label{sec:totalvariation}

In this section we give the proof of Theorem \ref{thm:main1}. The following result is the key ingredient for the proof. 

\begin{result}\label{re:sourav}\cite[ Theorem 2.2]{chatterjee}
{\it Let $X=(X_1,X_2,\ldots,X_n)$ be a vector of independent random variables in $\mathcal L(c_1,c_2)$ for some finite $c_1,c_2$. Take any $g\in C^2(\mathbb R^n)$ and let $\nabla g$ and $\nabla^2 g$ denote the gradient and Hessian of $g$. Let 
\begin{align*}
\kappa_0= \l(\E\sum_{k=1}^n\l|\frac{\partial g}{\partial x_k}(X)\r|^4\r)^{\frac{1}{2}},\ 
\kappa_1= (\E\|\nabla g(X) \|^4)^{\frac{1}{4}} \mbox{ and }
\kappa_2= (\E\|\nabla^2 g(X) \|^4)^{\frac{1}{4}}.
\end{align*}
Suppose $W=g(X)$ has a finite fourth moment and  $\sigma^2=\var(W)$. Let $Z$ be a normal random variable having the same mean and variance as $W$. Then 
$$
d_{TV}(W,Z)\le \frac{2\sqrt{5}(c_1c_2\kappa_0+c_1^3\kappa_1\kappa_2)}{\sigma^2}.
$$}
\end{result} 

%For the proof of Result \ref{re:sourav} we refer to the proof of Theorem 2.2 in \cite{chatterjee}.
 We use Result \ref{re:sourav} to prove Theorem \ref{thm:main1}, and for that we need to estimate  $\kappa_0, \kappa_1,\kappa_2$ and $\sigma^2$. The following  lemma gives the estimates of these quantities.

\begin{lemma}\label{lem:kappa}
Let $g(X_0,X_1,\ldots,X_{n-1})=\Tr(P_d(C_n))$ and consider $\kappa_0,\kappa_1$ and $\kappa_2$ as defined in Result \ref{re:sourav}. Then
\begin{align*}
\kappa_0&= O(n^{\frac{1}{2}}),\;\;
\kappa_1= O(n^{\frac{1}{2}})\;\;\mbox{ and }\;\;
\kappa_2=O\l(\frac{1}{n}(\sqrt{ \log n})^{d-2}\r).
\end{align*}
\end{lemma}

\noindent Assuming Lemma \ref{lem:kappa} we proceed to prove Theorem \ref{thm:main1}. 

\begin{proof}[Proof of Theorem \ref{thm:main1}] Let $W_n=\Tr(P_d(C_n))$. Using  Lemma \ref{lem:kappa} in Result \ref{re:sourav}, we  get
\begin{align}\label{eqn:TVnorm}
d_{TV}(W_n,Z_n)\le \frac{O(\sqrt n )}{\var(\Tr(P_d(C_n)))},
\end{align}
where $Z_n$ is a normal random variable having the same mean and variance as $W_n$.
Now from the variance calculation \eqref{liming variance} in the proof of Theorem \ref{thm:cirpoly}, we get
\begin{align*}
%\lim_{n\to \infty}\var\l(\frac{\Tr(P_d(C_n))-\E[\Tr(P_d(C_n))]}{\sqrt n}\r)=\sigma_{P_d}^2,\, i.e.,\,
\lim_{n\to \infty}\frac{1}{n}\var(\Tr(P_d(C_n)))=\sigma_{P_d}^2.
\end{align*}
Which implies that the right hand side of \eqref{eqn:TVnorm} goes to zero as $n\to \infty$, as $\sigma_{P_d}^2>0$. Hence the result.
\end{proof}

\noindent It remains to prove Lemma \ref{lem:kappa}. The following result will be used for estimating $\kappa_2$.
\begin{result}\label{res:norm1}
Let $C_n$ be a circulant matrix with input sequence $\{\frac{X_i}{\sqrt n}\}$, where $X_i$'s are symmetric $\sigma$-subgaussian. Then, for some $\alpha>0$, 
$$
\| C_n \|\le \alpha\sqrt{\log n} \;\;\mbox{a.s.},
$$
where $\|C_n\|:=\sup \{\|C_n x\|_2: x\in \R^n\}$ and $\|x\|_2=\sqrt{\sum_{i=1}^nx_i^2}$ for $x=(x_1,\ldots,x_n)^t\in~\R^n$.
\end{result}
We skip the proof of Result \ref{res:norm1}. For a proof of it  see the proof of Theorem 8 and Remark 19 in \cite{adhikari_saha2017}, and see \cite{meckes} also. The following result from \cite{chatterjee} will be used in the proof of Lemma \ref{lem:kappa}.

\begin{result}\label{re:norm}
{\it Let $A=(a_{ij})_{1\le i,j\le n}$ be an arbitrary square matrix with complex entries. Let $f(z)=\sum_{m=0}^{\infty}b_mz^m$ be an entire function. Define two associate entire functions $f_1=\sum_{m=1}^{\infty}m|b_m|z^{m-1}$ and $f_2=\sum_{m=2}^{\infty}m(m-1)|b_m|z^{m-2}$. Then, for each  $i,j$, we have 
$$\frac{\partial}{\partial a_{ij}}\Tr(f(A))=(f'(A))_{ji},$$
Next, for each $1\le i,j,k,\ell\le n$, let 
$$
h_{ij,k\ell}=\frac{\partial^2}{\partial a_{ij}\partial a_{k\ell}}\Tr(f(A)).
$$
Let $H$ be the $n^2\times n^2$ matrix $(h_{ij,k\ell})_{1\le i,j,k,\ell\le n}$. Then $\|H\|\le f_2(\|A\|)$.}
\end{result}
For the proof of Result \ref{re:norm}, we refer to \cite[Lemma 5.4]{chatterjee}. We use the following notation: For positive integers $p$ and $q$, define
\begin{align*}
N_{p}^{q}&=\{(i_1,i_2,\ldots,i_p)\suchthat i_1+i_2+\cdots + i_p=q,  0\le i_1,i_2,\ldots,i_p\le n-1\}.
\end{align*}
\begin{proof}[Proof of Lemma \ref{lem:kappa}]
Let $g(X_0,X_1,\ldots,X_{n-1})=\Tr(P_d(C_n))$. Then from the trace formula \eqref{trace formula C_n} of $C_n$, we have  
$$
g(X)=\sum_{k=2}^{d}\frac{a_k}{n^{\frac{k}{2}-1}}\sum_{A_k}X_{i_1}X_{i_2}\cdots X_{i_k}=\sum_{k=2}^{d}\frac{a_k}{n^{\frac{k}{2}-1}}\sum_{s=0}^{k-1}\sum_{N_{k}^{sn}}X_{i_1}X_{i_2}\cdots X_{i_k},
$$
where $X=(X_0,X_1,\ldots,X_{n-1})$. Therefore, for $0\le j,\ell\le n-1$, we have 
\begin{align*}
\frac{\partial g}{\partial x_j}(X)&=\sum_{k=2}^{d}\frac{a_k}{n^{\frac{k}{2}-1}}\sum_{s=0}^{k-1}k \sum_{N_{k-1}^{sn-j}}X_{i_1}X_{i_2}\cdots X_{i_{k-1}} \;\mbox{ and } \\\frac{\partial^2 g}{\partial x_{\ell}\partial x_j}(X)&=\sum_{k=2}^{d}\frac{a_k}{n^{\frac{k}{2}-1}}\sum_{s=0}^{k-1}k(k-1)\sum_{N_{k-2}^{sn-j-\ell}}X_{i_1}X_{i_2}\cdots X_{i_{k-2}}.
\end{align*}
Therefore we have 
\begin{align}\label{derivative of g power four}  
\E\l|\frac{\partial g}{\partial x_j}(X)\r|^4=\sum_{I_4}\frac{k_1k_2k_3k_4a_{k_1}a_{k_2}a_{k_3}a_{k_4}}{n^{\frac{k_1+k_2+k_3+k_4}{2}-4}}\sum_{S(k_1,k_2,k_3,k_4)}\sum_{N_{k_1,\ldots,k_4}^{s_1,\ldots,s_4}}\E\prod_{j=1}^4[X_{i_{j,1}}\cdots X_{i_{j,k_j-1}}].
\end{align}
where 
\begin{align*}I_4&=\{(k_1,\ldots,k_4)\; :\; 2\le k_1,\ldots,k_4\le d\},\\
 S(k_1,k_2,k_3,k_4)&=\{(s_1,\ldots,s_4): 0\le s_j\le k_j-1,j=1,\ldots,4\},\\
 N_{k_1,\ldots,k_4}^{s_1,\ldots,s_4}&=(N_{k_1-1}^{ s_1n-j},N_{k_2-1}^{ s_2n-j},N_{k_3-1}^{ s_3n-j},N_{k_4-1}^{ s_4n-j}).
 \end{align*} 
The input random variables are independent and have  mean zero, as they are symmetric $\sigma$-subgaussian. Therefore each random variable has to appear at least twice for non zero contribution in the right hand side of \eqref{derivative of g power four}. Note that, the total number of variables in the set $N_{k_1,\ldots,k_4}^{s_1,\ldots,s_4}$ is $k_1+k_2+k_3+k_4-4$. Following the  arguments as given  to find the limiting variance in the proof of Theorem \ref{thm:cirpoly}, we get
\begin{align*}
\sum_{N_{k_1,\ldots,k_4}^{s_1,\ldots,s_4}}\E\prod_{j=1}^4[X_{i_{j,1}}\cdots X_{i_{j,k_j-1}}]=O(n^{\frac{k_1+\cdots+k_4-4}{2}-2})=O(n^{\frac{k_1+\cdots+k_4}{2}-4}),
\end{align*} 
as the input random variables are $\sigma$-subgaussian. Since the degree $d$ of the polynomial  is fixed, we have 
\begin{align*}
\E\l|\frac{\partial g}{\partial x_j}(X)\r|^4=O(1) \;\; \mbox{and }\kappa_0=\l(\E\sum_{k=0}^{n-1}\l|\frac{\partial g}{\partial x_j}(X)\r|^4\r)^{\frac{1}{2}}=O(n^{\frac{1}{2}}).
\end{align*} 
Using Cauchy-Schwarz inequality and the bound of $\kappa_0$, we have
%Now from the bound on $\kappa_0$ with the use of Cauchy-Schwarz inequality, we have
%Using similar argument  we have 
\begin{align*}
\kappa_1=\l(\E\|\nabla g\|^4\r)^{\frac{1}{4}}=\l(\E\l(\sum_{k=0}^{n-1}\l|\frac{\partial g}{\partial x_k}(X)\r|^2\r)^2\r)^{\frac{1}{4}}=O(n^{\frac{1}{2}}).
\end{align*}
Now  we use Result \ref{re:norm} to get an upper bound for $\kappa_2$. Let $f(z)=P_d(z)$ and $A=C_n$. Then $a_{ij}=\frac{1}{\sqrt n}X_{j-i (\mbox{ mod $n$})}$, in particular, $a_{1i}=\frac{1}{\sqrt n} X_{i-1}$ for $i=1,\ldots, n$. Considering the matrix $A$ as a $n^2\times 1$ vector $(a_{11},\ldots, a_{1n},a_{21},\ldots, a_{2n},a_{31}, \ldots, a_{nn})^t$, the matrix  $H=(h_{ij,k\ell})$, where $h_{ij,k\ell}=\frac{\partial^2}{\partial a_{ij}\partial a_{k\ell}}\Tr(P_d(A)),$ has the following form
$$
 H=\l(\begin{array}{cc}
n[\nabla^2g]_{n\times n} & *
\\ * & *
\end{array}\r)_{n^2\times n^2}.
$$
Note that $n$ appeared in the first $n\times n$ block of $H$ due to the change of variables from $\{a_{11}, \ldots, a_{1n}\}$ to $\{x_0,\ldots, x_{n-1}\}$. From Results \ref{re:norm} and  \ref{res:norm1}, we have
 $$
 \|\nabla^2g\|\le \frac{1}{n}\|H\|\le \frac{1}{n} f_2(\|C_n\|)\le  C\frac{1}{n}(\sqrt{\log n})^{d-2} \;\;\mbox{ a.s.}
 $$  
for some non-random constant $C$.
%where last inequality follows from Result \ref{res:norm1}.
% and the first inequality follows from the fact that 
Therefore we have  
$$
\kappa_2=(\E\|\nabla^2 g(X)\|^4)^{\frac{1}{4}}=O\l(\frac{1}{n}(\sqrt{\log n})^{d-2}\r).
$$
This completes the proof.
\end{proof}

\noindent{\bf Acknowledgement:} We would like to thank Prof. Arup Bose for his comments. We thank both the referees for their useful suggestions.

%\bibliography{bibtex}
\bibliographystyle{amsplain}

\providecommand{\bysame}{\leavevmode\hbox to3em{\hrulefill}\thinspace}
\providecommand{\MR}{\relax\ifhmode\unskip\space\fi MR }
\providecommand{\MRhref}[2]{%
  \href{http://www.ams.org/mathscinet-getitem?mr=#1}{#2}
}
\providecommand{\href}[2]{#2}

\end{document}

%% file: mypreamble.tex
\theoremstyle{theorem}
    \newtheorem{theorem}{Theorem}
    \newtheorem{lemma}[theorem]{Lemma}

\theoremstyle{definition} % For roman text in the body
    \newtheorem{definition}[theorem]{Definition}
    
    \newtheorem{result}[theorem]{Result}
    \newtheorem{remark}[theorem]{Remark}
    \newtheorem{example}[theorem]{Example}
    \newtheorem{exercise}[theorem]{Exercise}

\def\suchthat{\; : \;}

\def\Z{\mathbb{Z}}

\def\R{\mathbb{R}}

\def\Z{\mathbb{Z}}

\def\l{\left}
\def\r{\right}
\def\<{\langle}
\def\>{\rangle}

\newcommand{\E}{\mbox{\bf E}}

\def\var{\mbox{Var}}

\newcommand\Tr{{\mbox{Tr}}}
\newcommand\tr{{\mbox{tr}}}

\newcommand\mnote[1]{} %off
\newcommand\be{\begin{equation*}}

\newcommand\ee{\end{equation*}}

\newcommand\ben{\begin{equation}}
\newcommand\een{\end{equation}}
\newcommand\bes{\begin{eqnarray*}}
\newcommand\ees{\end{eqnarray*}}

\newcommand\bex{\begin{exercise}}
\newcommand\eex{\end{exercise}}
\newcommand\beg{\begin{example}}
\newcommand\eeg{\end{example}}
\newcommand\benu{\begin{enumerate}}
\newcommand\eenu{\end{enumerate}}
\newcommand\beit{\begin{itemize}}
\newcommand\eeit{\end{itemize}}
\newcommand\berk{\begin{remark}}
\newcommand\eerk{\end{remark}}
\newcommand\bdefn{\begin{defintion}}
\newcommand\edefn{\end{definition}}
\newcommand\bthm{\begin{theorem}}
\newcommand\ethm{\end{theorem}}
\newcommand\bprf{\begin{proof}}
\newcommand\eprf{\end{proof}}
\newcommand\blem{\begin{lemma}}
\newcommand\elem{\end{lemma}}

\newcommand{\sm}{{\raise0.3ex\hbox{$\scriptstyle \setminus$}}}

\def\l{\left}
\def\r{\right}

\def\CHI{\mathchoice%
{\raise2pt\hbox{$\chi$}}%
{\raise2pt\hbox{$\chi$}}%
{\raise1.3pt\hbox{$\scriptstyle\chi$}}%
{\raise0.8pt\hbox{$\scriptscriptstyle\chi$}}}
\def\smalloplus{\raise1pt\hbox{$\,\scriptstyle \oplus\;$}}

%% file: CLT_circular28January,2018.bbl
\begin{thebibliography}{10}

\bibitem{adhikari_saha2017}
Kartick Adhikari and Koushik Saha, \emph{Fluctuations of eigenvalues of
  patterned random matrices}, J. Math. Phys. \textbf{58} (2017), no.~6.

\bibitem{anderson2006clt}
Greg~W Anderson and Ofer Zeitouni, \emph{A clt for a band matrix model},
  Probability Theory and Related Fields \textbf{134} (2006), no.~2, 283--338.

\bibitem{arharov}
L.~V. Arharov, \emph{Limit theorems for the characteristic roots of a sample
  covariance matrix}, Dokl. Akad. Nauk SSSR \textbf{199} (1971), 994--997.

\bibitem{bai2004clt}
Zhidong Bai and Jack~W Silverstein, \emph{Clt for linear spectral
  statistics of large-dimensional sample covariance matrices}, The Annals of
  Probability \textbf{32} (2004), no.~1A, 553--605.

\bibitem{chatterjee}
Sourav Chatterjee, \emph{Fluctuations of eigenvalues and second order
  poincar{\'e} inequalities}, Probability Theory and Related Fields
  \textbf{143} (2009), no.~1-2, 1--40.

\bibitem{girko}
Vyacheslav~L. Girko, \emph{Theory of stochastic canonical equations. {V}ol.
  {I,II}}, Mathematics and its Applications, vol. 535, Kluwer Academic
  Publishers, Dordrecht, 2001.

\bibitem{jana2014}
I.~Jana, K.~Saha, and A.~Soshnikov, \emph{Fluctuations of linear eigenvalue
  statistics of random band matrices}, Theory Probab. Appl. \textbf{60} (2016),
  no.~3, 407--443.

\bibitem{johansson1998}
Kurt Johansson, \emph{On fluctuations of eigenvalues of random hermitian
  matrices}, Duke Mathematical Journal \textbf{91} (1998), no.~1, 151--204.

\bibitem{jonsson}
Dag Jonsson, \emph{Some limit theorems for the eigenvalues of a sample
  covariance matrix}, Journal of Multivariate Analysis \textbf{12} (1982),
  no.~1, 1--38.

\bibitem{li2013central}
Lingyun Li and Alexander Soshnikov, \emph{Central limit theorem for linear
  statistics of eigenvalues of band random matrices}, Random Matrices: Theory
  and Applications \textbf{2} (2013), no.~04.

\bibitem{liu2012}
Dang-Zheng Liu, Xin Sun, and Zheng-Dong Wang, \emph{Fluctuations of eigenvalues
  for random toeplitz and related matrices}, Electron. J. Probab \textbf{17}
  (2012), no.~95, 1--22.

\bibitem{liu_wang2011}
Dang-Zheng Liu and Zheng-Dong Wang, \emph{Limit distribution of eigenvalues for
  random {H}ankel and {T}oeplitz band matrices}, J. Theoret. Probab.
  \textbf{24} (2011), no.~4, 988--1001. \MR{2851241}

\bibitem{lytova2009central}
A~Lytova and L~Pastur, \emph{Central limit theorem for linear eigenvalue
  statistics of random matrices with independent entries}, The Annals of
  Probability \textbf{37} (2009), no.~5, 1778--1840.

\bibitem{meckes}
Mark~W. Meckes, \emph{On the spectral norm of a random {T}oeplitz matrix},
  Electron. Comm. Probab. \textbf{12} (2007), 315--325 (electronic).

\bibitem{shcherbina2011central}
M~Shcherbina, \emph{Central limit theorem for linear eigenvalue statistics of
  the wigner and sample covariance random matrices}, Journal of Mathematical
  Physics, Analysis, Geometry \textbf{7} (2011), no.~2, 176--192.

\bibitem{shcherbina2015}
M.~Shcherbina, \emph{On fluctuations of eigenvalues of random band matrices},
  J. Stat. Phys. \textbf{161} (2015), no.~1, 73--90.

\bibitem{soshnikov1998tracecentral}
Ya. Sinai and A.~Soshnikov, \emph{Central limit theorem for traces of large
  random symmetric matrices with independent matrix elements}, Boletim da
  Sociedade Brasileira de MatemÃ¡tica - Bulletin/Brazilian Mathematical Society
  \textbf{29} (1998), no.~1, 1--24 (English).

\bibitem{stroock}
Daniel~W. Stroock, \emph{Probability theory, an analytic view}, 2nd edition, Cambridge
  University Press, Cambridge, 2011. 

\end{thebibliography}
